\documentclass[reqno]{amsart}
\usepackage{amsfonts,amsmath,amsthm,amssymb,mathrsfs}
\usepackage{color}
\usepackage{amsmath,amssymb,amsfonts,bm}
\usepackage{graphicx}
\usepackage{CJKutf8}
\usepackage{newunicodechar}
\usepackage{xcolor}
\numberwithin{equation}{section}
\usepackage[pdfstartview=FitH,colorlinks=true]{hyperref}
\hypersetup{urlcolor=red, linkcolor=blue,citecolor=red, CJKbookmarks=true}
\allowdisplaybreaks

\newtheorem{theorem}{Theorem}[section]

\newtheorem{lemma}[theorem]{Lemma}
\newtheorem{proposition}{Proposition}[section]

\theoremstyle{definition}

\newtheorem{remark}{Remark}[section]

\newcommand{\abs}[1]{\left\vert#1\right\vert}

\newcommand{\norm}[1]{\left\Vert#1\right\Vert}

\newcommand{\inner}[1]{\left(#1\right)}

\makeatletter

\newcommand{\Rmnum}[1]{\expandafter\@slowromancap\romannumeral #1@}
\makeatother

\title[Vanishing viscosity limit of Navier-Stokes Equations in Gevrey class]
{Vanishing viscosity limit of  Navier-Stokes \\
 Equations in Gevrey class}
\author{Feng Cheng, Wei-Xi Li and Chao-Jiang Xu}
\date{}

\address{\noindent \textsc{Feng Cheng, School of Mathematics and Statistics, Wuhan university 430072, Wuhan, P.R. China}}
\email{chengfengwhu@whu.edu.cn}

\address{\noindent \textsc{Wei-Xi Li, School of Mathematics and Statistics,    and Computational Science Hubei Key Laboratory, Wuhan university 430072, Wuhan, P.R. China}}
\email{wei-xi.li@whu.edu.cn}

\address{\noindent \textsc{Chao-Jiang Xu,
School of Mathematics and Statistics, Wuhan university 430072, Wuhan, P.R. China\\
and\\
Universit\'e de Rouen, CNRS UMR 6085, Laboratoire de Math\'ematiques, 76801 Saint-Etienne du Rouvray, France}}
\email{Chao-Jiang.Xu@univ-rouen.fr}

\begin{document}


\begin{abstract}
  In this paper we  consider the inviscid limit  for the periodic solutions to  Navier-Stokes equation in the  framework of Gevrey class.   It is shown that the lifespan for the solutions to Navier-Stokes equation is  independent of viscosity, and that the solutions of the Navier-Stokes equation converge to that of Euler equation in Gevrey class as the viscosity tends to zero. Moreover the convergence rate in Gevrey class is presented.
\end{abstract}

\maketitle

\section{Introduction}\label{section1}

The Navier-Stokes equations for incompressible viscous flow in $\mathbb{T}^3=(-\pi,\pi)^3$ read
\begin{equation}\label{1.1}
\left\{
\begin{aligned}
\frac{\partial u^\nu}{\partial t}-\nu \Delta u^\nu+(u^\nu\cdot\nabla)u^\nu+\nabla p^\nu&=0,\\
\nabla\cdot u^\nu&=0,\\
u^\nu|_{t=0}&=a,
\end{aligned}
\right.
\end{equation}
where $u^\nu(t,x)=(u^\nu_1,u^\nu_2,u^\nu_3)(t,x)$ is the unknown velocity vector function at point $x\in\mathbb{T}^3$ and  time $t$, $p^\nu(t,x)$ is the unknown scalar pressure function, $\nu>0$ is the kenematic viscosity, $a(x)=(a_1,a_2,a_3)(x)$ is the given initial data.

If the viscosity $\nu=0$, the equations \eqref{1.1} become the Euler equations for ideal flow with the same given initial data $a$,
\begin{equation}\label{1.2}
\left\{
\begin{aligned}
\frac{\partial u}{\partial t}+(u\cdot\nabla)u+\nabla p&=0,\\
\nabla\cdot u&=0,\\
u|_{t=0}&=a,
\end{aligned}
\right.
\end{equation}
where we denote the unknown vector velocity function to be $u(t,x)=(u_1,u_2,u_3)(t,x)$ and the unknown scalar pressure function to be $p(t,x)$.

The existence and uniqueness of  solutions to \eqref{1.1} and \eqref{1.2} in
Sobolev space $H^r(\mathbb{R}^3)$  for $r>3/2+1$, on a maximal time interval $[0,T_\ast)$ is classical in \cite{BB,TK,RT1}. There are abundant studies on the analyticities of solutions to \eqref{1.1} and \eqref{1.2} in various methods, for reference in \cite{BB1,BS,C,GK,samm}.
The Gevrey regularity of solutions to Navier-Stokes equations was started by Foias and Temam in their work \cite{FT}, in which the authors developed a way to prove the Gevrey class regularity by characterizing the decay of their Fourier coefficients. And later \cite{IK,IK1,KV,KV1,LO} developed this method to study the Gevrey class regularity of Euler equations in various conditions.

The subject of inviscid limits of solutions to Navier-Stokes equations has a long history and there is a vast literature on it, investigating this problem in various functional settings,  cf.  \cite{TK1,RT3} and references therein. Briefly, convergence of smooth solutions in $\mathbb R^n$ or torus  is well developed (cf.  \cite{TK,HS} for instance).  Much less is known about convergence   in a domain with boundaries.  In fact
the vanishing viscosity limit for the incompressible Navier-Stokes equations, in the case where there exist physical boundaries, is still a challenging problem due to the appearance of the Prandtl boundary layer which is caused by the classical no-slip boundary condition.  So far the rigorous verification of the Prandtl boundary layer theory was achieved only for some specific settings, cf. \cite{awxy, e-2,GM, guo, lwx,oleinik-3, xin-zhang} for instance, not to mention the convergence to  Prandtl's equation  and Euler equations. Several  partial results on the inviscid limits, in the case of half-space, were given in \cite{samm}  by imposing analyticity on the initial data, and in \cite{maekawa} for vorticity admitting  compact support which is away from the boundary.   

On the other hand, the Prandtl boundary layer equation is  ill-posed in Sobolev space for many case (see \cite{e-2, GD, Liu-Y}), while the Sobolev space is the suitable function space for the energy theory of  fluid mechanic.  Since the verification of the Prandtl boundary layer theory meet the major obstacle in the setting of the Sobolev space, it will be interesting to expect the vanishing viscosity limit for the incompressible Navier-Stokes equations in the setting of Gevery space as sub-space of Sobolev space, see a series of works in this direction \cite{GM, lwx, LY}.  In fact, Gevrey space is an intermediate space between the space of analytic functions and the Sobolev space. On one hand, Gevrey functions enjoy similar properties  as analytic functions, and on the other hand, there are nontrivial Gevrey functions having compact support,  which is different from analytic functions.   As a preliminary attempt,  in this work we study  the vanishing viscosity limit of  the solution of Navier Stokes equation to the solution of Euler equation in Gevrey space.   Here we will concentrate on the torus, we hope this may give insights on the case when the domain has boundaries, which is a much more challenging problem.

We introduce the functions spaces as follows. We usually suppress the vector symbol for functions when no ambiguity arise. Let $\mathcal{L}^2(\mathbb{T}^3)$ be the vector function space
\begin{equation*}
\begin{aligned}
\mathcal{L}^2(\mathbb{T}^3)=\bigg\{u=\sum_{k\in\mathbb{Z}^3}\hat u_j e^{i \, k\cdot x};\,\,
                      \, \hat u_k=\overline{\hat u_{-k}},\,\,\, \hat u_0=0,\,\ &\,j\cdot\hat u_j=0, \\
                     & \norm{u}_{L^2}^2=\sum_{k\in\mathbb{Z}^3}\abs{\hat u_k}^2<\infty \bigg\},
\end{aligned}
\end{equation*}
where $\hat u_k$ is the $k$ th order Fourier coefficient of $u$, $i=\sqrt{-1}$. The condition $j\cdot\hat u_j=0$ means $\nabla\cdot u=0$ in the weak sense, so it is the standard $L^2$ space with the divergence free condition.
Let $\mathcal{H}^r(\mathbb{T}^3)$ be the vector periodic Sobolev space : for  $r\geq 1$,
\begin{equation*}
\begin{aligned}
\mathcal{H}^r(\mathbb{T}^3)=\bigg\{u=\sum_{k\in\mathbb{Z}^3}\hat u_j e^{i \, k\cdot x};\,\,
                     & \, \hat u_k=\overline{\hat u_{-k}},\,\,\, \hat u_0=0,
                     \,\,j\cdot\hat u_j=0\\
                     & \norm{u}_{{H}^r}^2=\sum_{k\in\mathbb{Z}^3}(1+\abs{k}^2)^{r}\abs{\hat u_k}^2<\infty \bigg\}.
\end{aligned}
\end{equation*}
Here the condition $j\cdot\hat u_j=0$ means $\nabla\cdot u=0$, so it is the standard Sobolev space $H^r$ with the divergence free condition.
Denote $\inner{\cdot,\cdot}$ the $L^2$ inner product of two vector functions.
Let us define the fractional differential operator $\Lambda=(-\Delta)^{1/2}$ and the
 exponential operator $e^{\tau\Lambda^{1/s}}$ as follows,
$$
\Lambda u=\sum_{j\in\mathbb{Z}^3}\abs{j}\hat u_j e^{i j\cdot x},\,\,\,
e^{\tau\Lambda^{1/s}}u=\sum_{j\in\mathbb{Z}^3}e^{\tau\abs{j}^{1/s}}\hat u_j e^{i j\cdot x}.
$$
The vector Gevrey space $\mathcal{G}^{s}_{r,\tau}$ for $s\geq1, \tau>0, r\in\mathbb{R}$ is
\begin{equation*}
\begin{aligned}
\mathcal{G}^{s}_{r,\tau}(\mathbb{T}^3)=\left\{ u\in \mathcal{H}^r(\mathbb{T}^3) ; \quad \norm{u}_{\mathcal{G}^{s}_{r,\tau}}^2=\sum_{j\in\mathbb{Z}^3}\abs{j}^{2r} e^{2\tau \abs{j}^{1/s}} \abs{\hat u_j}^2 <\infty \right\},
\end{aligned}
\end{equation*}
where the condition  $j\cdot \hat u_j=0$ means\, $\nabla\cdot u=0$, so  it is sub-space of the Sobolev space $\mathcal{H}^r(\mathbb{T}^3)$.

The following theorem is the main result of this paper.
\begin{theorem}\label{Theorem1}
Let $r>\frac{9}{2}, \tau_0>0,  s\ge 1 $. Assume that the initial data $ a\in \mathcal{G}^s_{r,\tau_0}(\mathbb{T}^3)$, then there exists $\nu_0>0$ and $T>0, \tau(t)>0$ is a decreasing function
such that, for any $0<\nu\le \nu_0$, the Navier-Stokes equations \eqref{1.1} admit the solutions
\begin{align*}
u^\nu\in L^\infty([0,T];\mathcal{G}^s_{r, \tau(\,\cdot\,)}(\mathbb{T}^3));\ \ p^\nu \in L^\infty([0,T];\mathcal{G}^s_{r+1,  \tau(\,\cdot\,)}(\mathbb{T}^3)),
\end{align*}
 and the Euler equations \eqref{1.2} admit the solution
 \begin{align*}
u\in L^\infty([0,T];\mathcal{G}^s_{r,\tau(\,\cdot\,)}(\mathbb{T}^3));\ \ p \in L^\infty([0,T];\mathcal{G}^s_{r+1,\tau(\,\cdot\,)}(\mathbb{T}^3)),
\end{align*}
Furthermore, we have the following convergence estimates : for any $ 0<t\leq T$
\begin{equation}\label{1.5}
 \norm{u^\nu(t,\cdot)-u(t,\cdot)}_{\mathcal{G}^s_{r-1,\tau(t)}}\leq C\sqrt{\nu},\quad \norm{p^\nu(t,\cdot)-p(t,\cdot)}_{\mathcal{G}^s_{r,\tau(t)}}\leq C\sqrt{\nu},
\end{equation}
where $C$ is a constant depending on $r,s,a$ and $T$.
\end{theorem}

\begin{remark}\label{remark}
The uniform lifespan is  $0<T<T_\ast$ where $T_\ast$ is the maximal lifespan of $\mathcal{H}^r$ solutions. The uniform (with respect to $\nu$)  Gevrey  radius $\tau(t)$ of the solution is
\begin{equation}\label{1.6}
\tau(t)=\frac{1}{e^{C_1 t}\frac{1}{\tau_0}+\frac{C_2}{C_1}(e^{C_1 t}-1)}
\end{equation}
where $C_1,C_2$ are constants depending on $r, s, a,T$.
\end{remark}
\begin{remark} 
Comparaison with the known works about Gevery regularity of  Navier-Stokes equations and Euler equations \cite{BB1, BS, FT, IK1, KV, KV1}, the additional difficulties of this work is the uniform estimate of Gevery norm with respect to viscosity coefficients, and the estimate \eqref{1.5} with limit rates $\sqrt{\nu}$. 
\end{remark}

The paper is organized as follows. In section \ref{section2}, we will give the known results and preliminary lemmas. Section \ref{section3} consists of a priori estimate and the existence of the solutions in Gevrey space. The convergence in Gevrey space will be given in section \ref{section4}.

\section{Premilinary lemmas}\label{section2}

We first recall the following classical result of Kato in \cite{TK}.
\begin{theorem}\label{th 2.1}
Let $a\in \mathcal{H}^m (\mathbb{T}^3)$ for $m\geq 3$, then the following holds.\\
(1).There exists $T>0$ depending on $\norm{a}_{H^m}$ but not on $\nu$, such that \eqref{1.1} has a unique solution
$$
 u^\nu \in C([0,T],\mathcal{H}^m(\mathbb{T}^3))\, .
$$
Furthermore, $\{u^\nu\}$ is bounded in $C([0,T],\mathcal{H}^m(\mathbb{T}^3))$ for all $\nu>0$.\\
(2).For each $t\in[0,T]$, $u(t)=\lim_{\nu\to 0}u^\nu(t)$ exists strongly in $\mathcal{H}^{m-1}(\mathbb{T}^3)$ and weakly in $\mathcal{H}^m(\mathbb{T}^3)$, uniformly in $t$. $u$ is the unique solution to \eqref{1.2} satisfying
$$
 u\in C([0,T],\mathcal{H}^m(\mathbb{T}^3)).
$$
\end{theorem}
\begin{remark}
The time $T$ in Theorem 2.1 is actually depending on $m$ and $\norm{a}_{H^m}$, specifically
$$
 T<\frac{1}{C_m \norm{a}_{H^m}},
$$
where $C_m$ is a constant depending on $m$. In fact, the constant $C_m$ was created by using the Leibniz formula and Sobolev embedding inequality when estimating the nonlinear term. So, if the initial data $a\in \mathcal{G}^s_{r,\tau_0}(\mathbb{T}^3)$, then we have $a\in H^m_\sigma$, $\forall m$, because there exists a constant $C_{m,\tau_0,s}$ such that $\norm{a}_{H^m}\leq C_{m,\tau_0,s}\norm{a}_{\mathcal{G}^s_{r,\tau_0}}$. But we can't directly obtain an uniform bound for $C_m \norm{a}_{H^m}$ by the Gevrey norm of $\norm{a}_{\mathcal{G}^s_{r,\tau_0}}$ when $m$ is very large. Then we can't say that, if $m$ goes to infinity, $\frac{1}{C_m \norm{a}_{H^m}}$ has a positive lower bound. In this paper, we will pay many attention to the uniform lifespan $T>0$ that depends on $r,\norm{a}_{H^r}$.
\end{remark}
\begin{remark}
Compared with the known results Theorem \ref{th 2.1}, the additional difficulty arises on the estimate of the convecting term in Gevrey class setting. We need to use the decaying property of the radius of Gevrey class regularity to cancel the growth of the convecting term.
\end{remark}
We will use the following inequality,  for any $j,k\in\mathbb{Z}^3\backslash\{0\}$, we have
\begin{equation*}
\abs{k-j}\leq {2}\abs{j}\abs{k}.
\end{equation*}
The proof is a simple result of triangle inequality which we omit the details here. And we will give two Lemmas which will be used in the proof of Theorem \ref{Theorem1}.
\begin{lemma}\label{Lemma2}
Given two real numbers $\xi,\eta\geq1$ and $s\geq1$, then the following inequality holds
\begin{equation}\label{2.1}
\abs{\xi^{\frac{1}{s}}-\eta^{\frac{1}{s}}}\leq C\frac{\abs{\xi-\eta}}{\abs{\xi}^{1-\frac{1}{s}}+\abs{\eta}^{1-\frac{1}{s}}}
\end{equation}
where $C$ is a positive constant depending only on $s$.
\end{lemma}
\begin{proof}
The case for $s=1$ is trivial. Let us consider the case for $s>1$.
Without loss of generality, we may assume $\xi>\eta$. Then \eqref{2.1} is equivalent with
\begin{equation*}
\frac{(\xi^{1\over s}-\eta^{1\over s})(\xi^{1-{1\over s}}+\eta^{1-{1\over s}})}{\xi-\eta} \leq C.
\end{equation*}
Then it suffices to show that
\begin{equation*}
\abs{\frac{(\frac{\eta}{\xi})^{1-{1\over s}} -(\frac{\eta}{\xi})^{1\over s}}{1-\frac{\eta}{\xi}}}\leq C.
\end{equation*}
By Theorem 42 in \cite{Ha}, it can be obtained for fixed $s>1$
\begin{equation*}
\abs{\frac{(\frac{\eta}{\xi})^{1-{1\over s}} -(\frac{\eta}{\xi})^{1\over s}}{1-\frac{\eta}{\xi}}}\leq \max\left(1-\frac{2}{s},\frac{2}{s}-1\right)\leq C
\end{equation*}
Then the lemma \ref{Lemma2} is proved.
\end{proof}

With the use of Lemma \ref{Lemma2}, we have the following estimate about the nonlinear term.
\begin{lemma}\label{Lemma3}
Let $r>{9\over2}, s\geq 1 $ and $\tau>0$ is a constant. Then for any $v\in \mathcal{G}^s_{r+1,\tau}(\mathbb{T}^3)$, the following estimate holds,
\begin{equation}\label{2.2}
\begin{split}
&\abs{\inner{\Lambda^r e^{\tau\Lambda^{1/s}}(v\cdot\nabla v),\, \Lambda^{r}e^{\tau\Lambda^{1/s}}v}} \leq C\norm{ v}_{H^r}\norm{v}_{\mathcal{G}^{s}_{r,\tau}}^2+C\norm{v}_{H^r}^2\norm{v}_{\mathcal{G}^{s}_{r,\tau}}\\
 &\qquad\qquad\quad+\left[C\tau \norm{u}_{H^r}+C\tau^2 (\norm{u}_{H^r}+\norm{u}_{\mathcal{G}^s_{r,\tau}})\right]\norm{u}_{\mathcal{G}^{s}_{r+\frac{1}{2s},\tau}}^2\, ,
\end{split}
\end{equation}
where $C$ is a constant depending only on $r$ and $s$.
\end{lemma}

\begin{proof}
By the definition of the vector function space $\mathcal{G}^s_{r+1,\tau}(\mathbb{T}^3)$, we have $v=\sum_{j\in\mathbb{Z}^3}\hat v_j e^{i j\cdot x}$ and $\hat v_0=0$.
Using Fourier series convolution property, one have
\begin{equation*}
\begin{aligned}
 v\cdot\nabla v &=i\sum_{k\in\mathbb{Z}^3}\sum_{j\in\mathbb{Z}^3} [\hat v_j \cdot (k-j)]\hat v_{k-j} e^{i k\cdot x}.
\end{aligned}
\end{equation*}
Applying the operator $\Lambda^r e^{\tau \Lambda^{1/s}} $ on $v\cdot\nabla v$, one have
$$
 \Lambda^r e^{\tau \Lambda^{1/s}}(v\cdot\nabla v)=i\sum_{k\in\mathbb{Z}^3}\sum_{j\in\mathbb{Z}^3}[\hat v_j\cdot(k-j)]\hat v_{k-j} \abs{k}^r e^{\tau \abs{k}^{1/s}} e^{i k\cdot x}.
$$
And $\Lambda^{r}e^{\tau \Lambda^{1/s}}v=\sum_{\ell\in\mathbb{Z}^3}\abs{\ell}^{r}e^{\tau \abs{\ell}^{1/s}}\hat v_\ell e^{i \ell\cdot x}$. Now we take the $L^2$ inner product of $\Lambda^r e^{\tau\Lambda^{1/s}}(v\cdot\nabla v)$ with $\Lambda^{r}e^{\tau \Lambda^{1/s}}v$ over $\mathbb{T}^3$. The orthogonality of the exponentials in $L^2$ implies
\begin{equation*}
\begin{split}
&\inner{\Lambda^{r}e^{\tau \Lambda^{1/s}}(v\cdot\nabla v), \Lambda^{r}e^{\tau \Lambda^{1/s}}v}\\
&= i(2\pi)^3\sum_{k\in\mathbb{Z}^3}\sum_{j\in\mathbb{Z}^3}\left[\hat{v}_j\cdot (k-j)\right](\hat{v}_{k-j}\cdot{\hat{v}_{-k}})\abs{k}^{2r}e^{2\tau\abs{k}^{1/s}}.
\end{split}
\end{equation*}
The cancellation property of the convecting term implies
\begin{equation*}
\begin{split}
 0&=\inner{v\cdot\nabla\Lambda^r e^{\tau \Lambda^{1/s}}v,\Lambda^r e^{\tau \Lambda^{1/s}}v}\\
  &=i(2\pi)^3 \sum_{k\in\mathbb{Z}^3}\sum_{j\in\mathbb{Z}^3}\left[\hat v_j\cdot (k-j)\right]\abs{k-j}^r e^{\tau \abs{k-j}^{1/s}}(\hat v_{k-j}\cdot\hat v_{-k})\abs{k}^r e^{\tau \abs{k}^{1/s}}.
\end{split}
\end{equation*}
Then we have
\begin{equation*}
\begin{split}
&\quad\inner{\Lambda^{r}e^{\tau \Lambda^{1/s}} (v\cdot\nabla v),\Lambda^{r}e^{\tau \Lambda^{1/s}} v}\\
&=\inner{\Lambda^r e^{\tau \Lambda^{1/s}}(v\cdot\nabla v)-v\cdot\nabla \Lambda^r e^{\tau \Lambda^{1/s}}v,\Lambda^r e^{\tau \Lambda^{1/s}}v}\\
&=\mathcal{T}_1+\mathcal{T}_2,
\end{split}
\end{equation*}
where
$$
\mathcal{T}_1=i(2\pi)^3 \sum_{k\in\mathbb{Z}^3}\sum_{j\in\mathbb{Z}^3}(\abs{k}^r-\abs{k-j}^r)e^{\tau \abs{k-j}^{1/s}}\left[\hat v_j\cdot (k-j)\right](\hat v_{k-j}\cdot\hat v_{-k}) \abs{k}^r e^{\tau \abs{k}^{1/s}},
$$
and
$$
\mathcal{T}_2=i(2\pi)^3 \sum_{k\in\mathbb{Z}^3}\sum_{j\in\mathbb{Z}^3}\abs{k}^{r}(e^{\tau \abs{k}^{1/s}}-e^{\tau \abs{k-j}^{1/s}})\left[\hat v_j\cdot (k-j)\right](\hat v_{k-j}\cdot\hat v_{-k}) \abs{k}^r e^{\tau \abs{k}^{1/s}}.
$$
Before we come to the estimate of $\mathcal{T}_1$ and $\mathcal{T}_2$, we recall the following mean value theorem, for $\forall \xi,\eta\in \mathbb{R}^+$, there exists a constant $0\leq\theta,\theta^\prime\leq1$ such that
\begin{equation*}
\begin{split}
\xi^r-\eta^r &=r(\xi-\eta)\left[(\theta\xi+(1-\theta)\eta)^{r-1}-\xi^{r-1}\right]+r(\xi-\eta)\eta^{r-1}\\
                       &= r(r-1)\theta(\xi-\eta)^2[\theta^\prime(\theta\xi+(1-\theta)\eta)+(1-\theta^\prime)\eta]^{r-2}\\
                       &\quad +r(\xi-\eta)\eta^{r-1}.\\
\end{split}
\end{equation*}
Then there exists a constant $C$ depending only on $r$ such that
\begin{align*}
 \abs{\abs{k}^r-\abs{k-j}^r} \leq C\abs{j}^2(\abs{j}^{r-2}+\abs{k-j}^{r-2})+C \abs{j}\abs{k-j}^{r-1}.
\end{align*}
From the inequality $e^\xi\leq e+\xi^2 e^\xi$ that holds for all $\xi\in \mathbb{R}$, we can bounded the exponential $e^{\tau\abs{k-j}^{1/s}}$ by $e+\tau^2\abs{k-j}^{2/s}e^{\tau\abs{k-j}^{1/s}}$.
Then $\mathcal{T}_1$ can be bounded by
\begin{equation*}
\begin{split}
 &\quad\abs{\mathcal{T}_1}\\
 &\leq C\sum_{k\in\mathbb{Z}^3}\sum_{j\in\mathbb{Z}^3}(\abs{j}^{r}+\abs{j}^2\abs{k-j}^{r-2})\abs{\hat v_j}\abs{k-j}\abs{\hat v_{k-j}}(e+\tau^2\abs{k-j}^{2/s}e^{\tau \abs{k-j}^{1/s}}) \\
          &\quad \times\abs{\hat v_{-k}}\abs{k}^r e^{\tau \abs{k}^{1/s}} +C\sum_{k\in\mathbb{Z}^3}\sum_{j\in\mathbb{Z}^3}\abs{j}\abs{\hat v_j}\abs{k-j}^r e^{\tau \abs{k-j}^{1/s}}\abs{\hat v_{k-j}}\abs{\hat v_{-k}}\abs{k}^r e^{\tau \abs{k}^{1/s}}\\
          &=\mathcal{T}_{11}+\mathcal{T}_{12}+\mathcal{T}_{13}+\mathcal{T}_{14}+\mathcal{T}_{15}.
\end{split}
\end{equation*}
With application of discrete H\"older inequality and Minkowski inequality, one can obtain the following estimates. For example, we give the details for $\mathcal{T}_{11}$, and the rest can be estimated in the same way,
\begin{align*}
 \mathcal{T}_{11} &=eC\sum_{k\in\mathbb{Z}^3}\sum_{j\in\mathbb{Z}^3}\abs{j}^r \abs{\hat v_j} \abs{k-j}\abs{\hat v_{k-j}} \abs{k}^r e^{\tau\abs{k}^{1/s}}\abs{\hat v_{-k}}\\
    &\leq C\norm{v}_{\mathcal{G}^s_{r,\tau}}\left[\sum_{k\in\mathbb{Z}^3}\left(\sum_{j\in\mathbb{Z}^3}\abs{j}^r \abs{\hat v_j}\abs{k-j}\abs{\hat v_{k-j}}\right)^2\right]^{1/2}\\
    &=C\norm{v}_{\mathcal{G}^s_{r,\tau}}\left[\sum_{k\in\mathbb{Z}^3}\left(\sum_{\ell\in\mathbb{Z}^3}\abs{k-\ell}^r \abs{\hat v_{k-\ell}}\abs{\ell}\abs{\hat v_{\ell}}\right)^2\right]^{1/2}\\
    &\leq C\norm{v}_{\mathcal{G}^s_{r,\tau}}\left(\sum_{k\in\mathbb{Z}^3}\abs{k-\ell}^{2r}\abs{\hat v_{k-\ell}}^2\right)^{1/2}\sum_{\ell\in\mathbb{Z}^3}\frac{\abs{\ell}}{(1+\abs{\ell}^2)^{r/2}}(1+\abs{\ell}^2)^{r/2}\abs{\hat v_\ell}\\
    &\leq C\norm{v}_{H^r}^2 \norm{v}_{\mathcal{G}^s_{r,\tau}} \left(\sum_{\ell\in\mathbb{Z}^3}\frac{\abs{\ell}^2}{(1+\abs{\ell}^2)^r}\right)^{1/2}\\
    &\leq C\norm{v}_{H^r}^2 \norm{v}_{\mathcal{G}^s_{r,\tau}},
\end{align*}
where $C$ is a constant depending on $r,e$ and for $r>9/2$, the summation in the above $\left(\sum_{\ell\in\mathbb{Z}^3}\frac{\abs{\ell}^2}{(1+\abs{\ell}^2)^r}\right)^{1/2}$ is bounded by some constant depending on $r$. Similarly with $\mathcal{T}_{11}$, we have
\begin{align*}
 \mathcal{T}_{12} &=eC\sum_{k\in\mathbb{Z}^3}\sum_{j\in\mathbb{Z}^3}\abs{j}^2 \abs{\hat v_j} \abs{k-j}^{r-1}\abs{\hat v_{k-j}} \abs{k}^r e^{\tau\abs{k}^{1/s}}\abs{\hat v_{-k}}\\
   &\leq C\norm{v}_{H^r}^2\norm{v}_{\mathcal{G}^s_{r,\tau}},
\end{align*}
and
\begin{align*}
 \mathcal{T}_{13} &=C\tau^2\sum_{k\in\mathbb{Z}^3}\sum_{j\in\mathbb{Z}^3}\abs{j}^r \abs{\hat v_j} \abs{k-j}^{1+2/s}\abs{\hat v_{k-j}} \abs{k}^r e^{\tau\abs{k}^{1/s}}\abs{\hat v_{-k}}\\
  &\leq C\tau^2 \norm{v}_{H^r}\norm{v}_{\mathcal{G}^s_{r,\tau}}^2.
\end{align*}
Note that $\hat v_0=0, s\geq 1$ in the summation, and $\abs{k-j}^{\frac{1}{2s}}\leq C\abs{k}^{\frac{1}{2s}}\abs{j}^{\frac{1}{2s}}$, we can similarly have
\begin{align*}
 \mathcal{T}_{14} &=C\tau^2\sum_{k\in\mathbb{Z}^3}\sum_{j\in\mathbb{Z}^3}\abs{j}^2 \abs{\hat v_j} \abs{k-j}^{r-1+2/s}e^{\tau\abs{k-j}^{1/s}}\abs{\hat v_{k-j}} \abs{k}^r e^{\tau\abs{k}^{1/s}}\abs{\hat v_{-k}}\\
  &\leq C\tau^2\sum_{k\in\mathbb{Z}^3}\sum_{j\in\mathbb{Z}^3}\abs{j}^{2+\frac{1}{2s}} \abs{\hat v_j} \abs{k-j}^{r+\frac{1}{2s}}e^{\tau\abs{k-j}^{1/s}}\abs{\hat v_{k-j}} \abs{k}^{r+\frac{1}{2s}} e^{\tau\abs{k}^{1/s}}\abs{\hat v_{-k}}\\
  &\leq C\tau^2 \norm{v}_{H^r}\norm{v}_{\mathcal{G}^s_{r+\frac{1}{2s},\tau}}^2,
\end{align*}
and
\begin{align*}
 \mathcal{T}_{15} &=C\sum_{k\in\mathbb{Z}^3}\sum_{j\in\mathbb{Z}^3}\abs{j} \abs{\hat v_j} \abs{k-j}^r e^{\tau\abs{k-j}^{1/s}}\abs{\hat v_{k-j}} \abs{k}^r e^{\tau\abs{k}^{1/s}}\abs{\hat v_{-k}}\\
 &\leq C\norm{v}_{H^r}\norm{v}_{\mathcal{G}^s_{r,\tau}}^2.
\end{align*}
Noting that $\norm{v}_{\mathcal{G}^s_{r,\tau}}\leq \norm{v}_{\mathcal{G}^s_{r+\frac{1}{2s},\tau}}$, then $\mathcal{T}_{13}\leq \mathcal{T}_{14}$. Thus we obtain
$$
 \abs{\mathcal{T}_1}\leq C\norm{v}_{H^r}^2\norm{v}_{\mathcal{G}^s_{r,\tau}}+C\norm{v}_{H^r} \norm{v}_{\mathcal{G}^s_{r,\tau}}^2+C\tau^2 \norm{v}_{H^r}\norm{v}_{\mathcal{G}^s_{r+\frac{1}{2s},\tau}}^2.
$$
As for $\mathcal{T}_2$, we have
\begin{equation*}
\begin{split}
 \mathcal{T}_2 &=i(2\pi)^3 \sum_{k\in\mathbb{Z}^3}\sum_{j\in\mathbb{Z}^3}\abs{k}^{r}(e^{\tau \abs{k}^{1/s}}-e^{\tau \abs{k-j}^{1/s}})[\hat v_j\cdot (k-j)](\hat v_{k-j}\cdot\hat v_{-k}) \abs{k}^r e^{\tau \abs{k}^{1/s}}\\
      &=i(2\pi)^3 \sum_{k\in\mathbb{Z}^3}\sum_{j\in\mathbb{Z}^3}\abs{k}^r e^{\tau \abs{k-j}^{1/s}}\left[e^{\tau (\abs{k}^{1/s}-\abs{k-j}^{1/s})}-1 \right][\hat v_j\cdot (k-j)]\\
      &\quad\times(\hat v_{k-j}\cdot\hat v_{-k}) \abs{k}^r e^{\tau \abs{k}^{1/s}}.
\end{split}
\end{equation*}
We note that the inequality $\abs{e^\xi-1}\leq \abs{\xi}e^{\abs{\xi}}$ holds for $\xi\in\mathbb{R}$. Then
$$
 \abs{e^{\tau (\abs{k}^{1/s}-\abs{k-j}^{1/s})}-1}\leq \tau \abs{\abs{k}^{1/s}-\abs{k-j}^{1/s}}e^{\tau \abs{\abs{k}^{1/s}-\abs{k-j}^{1/s} }}.
$$
Since $s\geq1$, we have
$$
 \abs{\abs{k}^{1/s}-\abs{k-j}^{1/s}}\leq \abs{j}^{1/s}.
$$
Then we actually have
$$
 \abs{e^{\tau (\abs{k}^{1/s}-\abs{k-j}^{1/s})}-1}\leq \tau \abs{\abs{k}^{1/s}-\abs{k-j}^{1/s}}e^{\tau \abs{j}^{1/s}}.
$$
By Lemma \ref{Lemma2}, we have
$$
\abs{\abs{k}^{1/s}-\abs{k-j}^{1/s}}\leq C\abs{j}\frac{1}{\abs{k}^{1-1/s}+\abs{k-j}^{1-1/s}}.
$$
Then $\mathcal{T}_2$ can be bounded by the inequality
\begin{equation*}
\begin{split}
\abs{\mathcal{T}_2} &\leq C\sum_{k\in\mathbb{Z}^3}\sum_{j\in\mathbb{Z}^3}\abs{k}^r e^{\tau \abs{k-j}^{1/s}}\abs{e^{\tau (\abs{k}^{1/s}-\abs{k-j}^{1/s})}-1}\abs{k-j}\abs{\hat v_j}\abs{\hat v_{k-j}}\\
    &\quad\times\abs{v_{-k}}\abs{k}^r e^{\tau \abs{k}^{1/s}} \\
    &\leq C\sum_{k\in\mathbb{Z}^3}\sum_{j\in\mathbb{Z}^3}\abs{k}^{r-\frac{1}{2s}} e^{\tau \abs{k-j}^{1/s}}\tau \abs{\abs{k}^{1/s}-\abs{k-j}^{1/s}}e^{\tau \abs{j}^{1/s}}\abs{k-j}\abs{\hat v_j}\abs{\hat v_{k-j}}\\
    &\quad\times\abs{\hat v_{-k}}\abs{k}^{r+\frac{1}{2s}} e^{\tau \abs{k}^{1/s}} \\
    &\leq C\tau\sum_{k\in\mathbb{Z}^3}\sum_{j\in\mathbb{Z}^3}(\abs{j}^{r-\frac{1}{2s}}+\abs{k-j}^{r-\frac{1}{2s}}) e^{\tau \abs{k-j}^{1/s}}\frac{\abs{j}\abs{k-j}}{\abs{k}^{1-1/s}+\abs{k-j}^{1-1/s}} \\
    &\quad\times e^{\tau \abs{j}^{1/s}}\abs{\hat u_j}\abs{\hat v_{k-j}}\abs{\hat v_{-k}}\abs{k}^{r+\frac{1}{2s}} e^{\tau \abs{k}^{1/s}}\\
    &\leq \mathcal{T}_{21}+\mathcal{T}_{22},
\end{split}
\end{equation*}
where
\begin{align*}
 \mathcal{T}_{21} &=C\tau\sum_{k\in\mathbb{Z}^3}\sum_{j\in\mathbb{Z}^3}\abs{j}^{r+\frac{1}{2s}}e^{\tau\abs{j}^{1/s}}\abs{\hat v_j}\abs{k-j}(1+\tau\abs{k-j}^{1/s}e^{\tau\abs{k-j}^{1/s}})\\
 &\quad\times\abs{\hat v_{k-j}} \abs{k}^{r+\frac{1}{2s}}e^{\tau\abs{k}^{1/s}}\abs{\hat v_{-k}},\\
 \mathcal{T}_{22}&=C\tau\sum_{k\in\mathbb{Z}^3}\sum_{j\in\mathbb{Z}^3}\abs{j}(1+\tau\abs{j}^{1/s}e^{\tau\abs{j}^{1/s}})\abs{\hat v_j}\abs{k-j}^{r+\frac{1}{2s}}e^{\tau\abs{k-j}^{1/s}}\\
 &\quad\times\abs{\hat v_{k-j}} \abs{k}^{r+\frac{1}{2s}}e^{\tau\abs{k}^{1/s}}\abs{\hat v_{-k}}.
\end{align*}
We have used the inequality $\abs{k}^{1-1/s}+\abs{k-j}^{1-1/s}\geq \abs{j}^{1-1/s}$ and $e^\xi\leq 1+\xi e^\xi$ for $\xi\in\mathbb{R}^+$ in the estimation of $\mathcal{T}_{21}$.
With application of H{\"o}lder inequality and Minkowiski inequality, we have for $\mathcal{T}_{21}$,
\begin{align*}
\abs{\mathcal{T}_{21}} \leq C\tau \norm{v}_{H^r}\norm{v}_{\mathcal{G}^s_{r+\frac{1}{2s},\tau}}^2+
C\tau^2\norm{v}_{\mathcal{G}^s_{r,\tau}}\norm{v}_{\mathcal{G}^s_{r+\frac{1}{2s},\tau}}^2.
\end{align*}
Symmetrically, one has a same bound for $\mathcal{T}_{22}$, then for $\mathcal{T}_2$,
\begin{align*}
\abs{\mathcal{T}_{2}} \leq C\tau \norm{v}_{H^r}\norm{v}_{\mathcal{G}^s_{r+\frac{1}{2s},\tau}}^2+
C\tau^2\norm{v}_{\mathcal{G}^s_{r,\tau}}\norm{v}_{\mathcal{G}^s_{r+\frac{1}{2s},\tau}}^2.
\end{align*}
Then we obtain
\begin{align*}
\abs{\inner{\Lambda^{r}e^{\tau\Lambda^{1/s}}(v\cdot\nabla v),\Lambda^{r}e^{\tau\Lambda^{1/s}}v}} &\leq \abs{\mathcal{T}_1}+ \abs{\mathcal{T}_2}\\
 &\leq C\norm{v}_{H^r}^2 \norm{v}_{\mathcal{G}^s_{r,\tau}}+C\norm{v}_{H^r}\norm{v}_{\mathcal{G}^s_{r,\tau}}^2\\
 &\quad+C\tau^2\norm{v}_{\mathcal{G}^s_{r,\tau}}\norm{v}_{\mathcal{G}^s_{r+\frac{1}{2s},\tau}}^2\\
 &\quad+C\tau(1+\tau)\norm{v}_{H^r}\norm{v}_{\mathcal{G}^s_{r+\frac{1}{2s},\tau}}^2\, ,
\end{align*}
which finishes the proof of  Lemma \ref{Lemma3}.
\end{proof}

\section{Uniform existence of solutions}\label{section3}
In this section, we will first show the existence of Gevrey class solutions $u^\nu$ to Navier Stokes equations \eqref{1.1}. And the existence of Gevrey class solution $u$ to Euler equations \eqref{1.2} can be obtained similarly. The method of the proof are based on Galerkin approximation. Before that, we first consider the following equivalent equation for Navier-Stokes equation,
\begin{equation}\label{PNS}
\begin{aligned}
 \frac{d}{dt}u^\nu +\nu A u^\nu +\mathbb{P}(u^\nu\cdot\nabla u^\nu) &=0,\\
  u^\nu|_{t=0} &=\mathbb{P}a.
\end{aligned}
\end{equation}
where $A=-\mathbb{P}\Delta$ is the well-known Stokes operator and $\mathbb{P}$ is the Leray projector which maps a vector function $v$ into its divergence free part $v_1$, such that $v=v_1+\nabla q$ and $\nabla\cdot v_1=0$, $q$ is a scalar function and $\inner{v_1,\nabla q}=0$. Similarly for Euler equation, we have the following equivalent form,
\begin{equation}\label{PE}
\begin{aligned}
 \frac{d}{dt}u^\nu +\mathbb{P}(u^\nu\cdot\nabla u^\nu) &=0,\\
  u^\nu|_{t=0} &=\mathbb{P}a.
\end{aligned}
\end{equation}
We then recall some properties of the Stokes operator $A$, which are known in [Chapter 4 in \cite{CF}].
\begin{proposition}\label{prop4}
The Stokes operator $A$ is symmetric and selfadjoint, moreover, the inverse of the Stokes operator, $A^{-1}$, is a compact operator in $\mathcal{L}^2$. The Hilbert theorem implies there exists a sequence of positive numbers $\lambda_j$ and an orthonormal basis ${w_j}$ of $\mathcal{L}^2$, which satisfies
\begin{align*}
 A w_j=\lambda_j w_j,\,\, 0<\lambda_1<\ldots<\lambda_j\leq \lambda_{j+1}\leq \ldots,\,\,\lim_{j\to\infty}\lambda_j=\infty.\\
\end{align*}
\end{proposition}
Moreover, in the case of $\mathbb{T}^3=(-\pi,\pi)^3$, the sequence of eigenvector functions $w_j^{,}s$ and eigenvalues $\lambda_j^{,}s$ are the sequences of functions $w_{k,j}$ and numbers $\lambda_{k,j}$,
$$  w_{k,j}(x)=\left(e_j-\frac{k_j k}{\abs{k}^2} \right) e^{i k\cdot x} ,\quad\ \lambda_{k,j}=\abs{k}^2,
$$
where $k=(k_1,k_2,k_3)\in\mathbb{Z}^3$, $k\neq 0$, $j=1,2,3$ and $\{e_j\}_{j=1,2,3}$ are the canonical basis in \ $\mathbb{R}^3$. So we know that each $w_j$ are not only in $\mathcal{L}^2$, but also in $\mathcal{G}^s_{r,\tau}$ for $\forall r>0$. Now we will show that there exists a solution to equation \eqref{PNS} for $a\in\mathcal{G}^s_{r,\tau}$ with $r>9/2, s\geq 1$, and $\tau(t)>0$ is a differentiable decreasing function of $t$. To this end, we first prove a priori estimate in the following Proposition.

\begin{proposition}\label{proposition1}
Let $r>9/2, s\geq1, a\in\mathcal{G}^s_{r,\tau_0}$ and $\tau(t)>0$ is a differentiable decreasing function of $t$ defined on $[0,T]$ with $\tau(0)=\tau_0>0$, where $0<T<T^\ast$ and $T^\ast$ is the maximal time of $H^r$ solution to \eqref{PNS} with respect to the initial data $a$. Let $u^\nu(t,x)\in L^\infty([0,T];\mathcal{G}^s_{r,\tau(\cdot)}(\mathbb{T}^3))\cap L^2([0,T];\mathcal{G}^s_{r+1,\tau(\cdot)}(\mathbb{T}^3))$ be the solution to \eqref{PNS}, then the following a priori estimates holds,
\begin{equation*}
\begin{aligned}
 & \norm{u^\nu(t,\cdot)}_{\mathcal{G}^s_{r,\tau(t)}} \leq G_T,\quad\,0<t\leq T,\\
 &\nu\int_0^t \norm{u^\nu(s,\cdot)}_{\mathcal{G}^s_{r+1,\tau(s)}}^2ds \leq M_T,\quad\,0<t\leq T\,.
\end{aligned}
\end{equation*}
With the same assumptions as above, let $u(t,x)\in L^\infty([0,T];\mathcal{G}^s_{r+1,\tau(\cdot)}(\mathbb{T}^3))$ be the solution to \eqref{PE}, we also have
\begin{equation*}
  \norm{u(t,\cdot)}_{\mathcal{G}^s_{r,\tau(t)}} \leq G_T,\quad\,0<t\leq T,
\end{equation*}
Furthermore the uniform radius $\tau(t)$ is given by
\begin{equation*}
 \tau(t)= \frac{1}{e^{C_1 t}\frac{1}{\tau_0}+\frac{C_2}{C_1}(e^{C_1 t}-1)},
\end{equation*}
where $C,C_1,C_2,G_T,M_T$ are constants depending on $a,r,s,T$.
\end{proposition}
\begin{proof}
Applying $\Lambda^r e^{\tau\Lambda^{1/s}}$ on both sides of \eqref{PNS} and taking the $L^2$ inner product of both sides with $\Lambda^{r} e^{\tau \Lambda^{1/s}} u^\nu$, one has
\begin{equation}\label{3.1}
\begin{split}
&\frac{1}{2}\frac{d}{dt}\norm{u^\nu(t,\cdot)}_{\mathcal{G}^{s}_{r,\tau(t)}}^2+
\nu\norm{u^\nu(t,\cdot)}_{\mathcal{G}^{s}_{r+1,\tau(t)}}^2\\
&\quad=\tau^\prime(t) \norm{u^\nu(t,\cdot)}_{\mathcal{G}^s_{r+\frac{1}{2s},\tau}}^2-\inner{\Lambda^r e^{\tau\Lambda^{1/s}}(u^\nu\cdot\nabla u^\nu),\Lambda^{r} e^{\tau \Lambda^{1/s}}u^\nu},\\
\end{split}
\end{equation}
where we use the fact that $\mathbb{P}$ commutes with $\Lambda^r e^{\tau\Lambda^{1/s}}$ and $\mathbb{P}$ is symmetric. Now we consider the right hand side of \eqref{3.1}.
By \eqref{2.2} in Lemma \ref{Lemma3}, we have from \eqref{3.1}, for convenience, we sometimes suppress the dependence of $u^\nu$ and $\tau$ in $t$,
\begin{equation}\label{3.2}
\begin{split}
&\frac{1}{2}\frac{d}{dt}\norm{u^\nu}_{\mathcal{G}^{s}_{r,\tau}}^2 +\nu\norm{u_\nu}_{\mathcal{G}^{s}_{r+1,\tau}}^2 \leq C\norm{u^\nu}_{H^r}\norm{u^\nu}_{\mathcal{G}^s_{r,\tau}}^2+C\norm{u^\nu}_{H^r}^2\norm{u^\nu}_{\mathcal{G}^s_{r,\tau}}\\
&+(\tau^\prime+C\tau\norm{u^\nu}_{H^r}+C\tau^2\norm{u^\nu}_{H^r}+
C\tau^2\norm{u^\nu}_{\mathcal{G}^s_{r,\tau}})\norm{u^\nu}_{\mathcal{G}^{s}_{r+\frac{1}{2s},\tau}}^2,
\end{split}
\end{equation}
where $C$ is a constant depending on $r,s$.
Now if the radius of Gevrey class $\tau(t)$ is smooth and decreaseing fast enough such that the following inequality holds,
\begin{equation}\label{3.3}
\tau^\prime+C\tau\norm{u^\nu}_{H^r}+C\tau^2\norm{u^\nu}_{H^r}+C\tau^2\norm{u^\nu}_{\mathcal{G}^s_{r,\tau}}\leq 0.
\end{equation}
Then \eqref{3.2} implies
\begin{equation}\label{3.4}
\frac{1}{2}\frac{d}{dt}\norm{u^\nu}_{\mathcal{G}^{s}_{r,\tau}}^2+\nu\norm{u^\nu}_{\mathcal{G}^{s}_{r+1,\tau}}^2\leq C\norm{u^\nu}_{H^r}\norm{u^\nu}_{\mathcal{G}^s_{r,\tau}}^2+C\norm{u^\nu}_{H^r}^2\norm{u^\nu}_{\mathcal{G}^s_{r,\tau}}.
\end{equation}
As $\nu>0$, it can be obtained directly from \eqref{3.4},
\begin{equation}\label{3.5}
\frac{d}{dt}\norm{u^\nu}_{\mathcal{G}^{s}_{r,\tau}}\leq C\norm{u^\nu}_{H^r}\norm{u^\nu}_{\mathcal{G}^s_{r,\tau}}+C\norm{u^\nu}_{H^r}^2,
\end{equation}
By Grownwall's inequality in \eqref{3.5}, we have for $0<t< T_\ast$,
\begin{equation}\label{3.6}
\norm{u^\nu(t)}_{\mathcal{G}^{s}_{r,\tau(t)}}\leq \tilde{g}(t)\left(\norm{a}_{\mathcal{G}^s_{r,\tau_0}}+\int_0^t C\tilde{g}(s)^{-1} \norm{u^\nu(s)}_{H^r}^2 ds\right)\triangleq A(t),
\end{equation}
where $\tilde{g}(t)=e^{\int_0^t C\norm{u^\nu(s,\cdot)}_{H^r}ds}$ and $T_\ast$ is the maximal time interval of $H^r$ solution. It has been known that $T_\ast$ is independent of $\nu$. Moreover, it follows the a priori estimate for $H^r$ solution for Navier Stokes equation in \cite{MB},
\begin{equation}\label{3.7}
\frac{d}{dt} \norm{u^\nu(t,\cdot)}_{H^r}^2 \leq C\norm{u^\nu(t,\cdot)}_{H^r}\norm{u^\nu(t,\cdot)}_{H^r}^2,\quad 0<t<T_\ast,
\end{equation}
where C is a constant depending on $r$. Then, on one side we have
\begin{equation}\label{3.8}
 \norm{u^\nu(t)}_{H^r}^2 \leq C \tilde{g}(t) \norm{a}_{H^r}^2,\quad 0<t<T_\ast.
\end{equation}
And on the other side, let $0<T<T_\ast$, then for $0<t<T$,
\begin{equation}\label{3.9}
 \norm{u^\nu(t,\cdot)}_{H^r} \leq \frac{\norm{a}_{H^r}}{1-Ct\norm{a}_{H^r}}\leq \frac{\norm{a}_{H^r}}{1-CT\norm{a}_{H^r}}\triangleq C_T.
\end{equation}
With \eqref{3.7},\eqref{3.8} and \eqref{3.9}, we have
\begin{equation}\label{Abound}
\begin{aligned}
 \norm{u^\nu(t,\cdot)}_{\mathcal{G}^s_{r,\tau(t)}} &\leq A(t)\\
  &=e^{\int_0^t C\norm{u^\nu(s)}_{H^r}ds}\left(\norm{a}_{G^s_{r,\tau_0}}+\int_0^t C\tilde{g}(s)^{-1} \norm{u^\nu(s)}_{H^r}^2 ds\right)\\
      &\leq e^{CC_T t}\left(\norm{a}_{G^s_{r,\tau_0}}+ C\norm{a}_{H^r}^2 t\right)\\
      &\leq e^{CC_T T}\left(\norm{a}_{G^s_{r,\tau_0}}+ C\norm{a}_{H^r}^2 T\right)\triangleq G_T,\quad 0\leq t\leq T.
\end{aligned}
\end{equation}
In fact a sufficient condition for \eqref{3.3} to hold is
\begin{equation}\label{3.10}
\tau^\prime(t)+C\tau(t) C_T+ C\tau^2(t)C_T +C\tau^2(t)G_T= 0.
\end{equation}
Then solving the ordinary differential equation \eqref{3.10},
\begin{equation}\label{3.11}
\begin{aligned}
 \tau(t) &= \frac{1}{e^{CC_T t}\frac{1}{\tau_0}+\frac{CC_T+CG_T}{CC_T}(e^{CC_T t}-1)}.\\
\end{aligned}
\end{equation}
We can obtain \eqref{1.6} by arranging the constants in \eqref{3.11},
\begin{equation}\label{3.12}
 \tau(t)= \frac{1}{e^{C_1 t}\frac{1}{\tau_0}+\frac{C_2}{C_1}(e^{C_1 t}-1)},
\end{equation}
where $C_1,C_2$ are constants depending on $r,s,T,a$. Then \eqref{3.12} proves \eqref{1.6} in Remark \ref{remark}.
Integrating \eqref{3.4} form $0$ to $t$, we have, for $0<t<T$,
\begin{equation}\label{3.13}
\nu\int_0^t \norm{u^\nu(s,\cdot)}_{\mathcal{G}^{s}_{r+1,\tau(s)}}^2ds\leq M_T,\ \ 0<t<T<T_\ast,
\end{equation}
where $M_T$ depends on $T,a,r,s,\tau_0$. It should be noted that all of the above estimates are independent of $\nu$, so we let $\nu=0$ in \eqref{3.2}, and proceed exactly as above, then we can obtain similar results for the a priori estimate for solution $u$ to equation \eqref{PE}.
\end{proof}

With the estimates in Proposition \ref{proposition1}, we can implement the standard Faedo-Fourier-Galerkin approximation as in \cite{JL,RT} to prove the existence of such $u^\nu$ and $u$ in the function space of Gevrey class s $\mathcal{G}^s_{r,\tau}$.

\begin{theorem}\label{existence}
There exists a unique solution $u^\nu$ to \eqref{PNS} such that
\begin{equation*}
 u^\nu \in L^\infty([0,T],\mathcal{G}^s_{r,\tau(\,\cdot\,)}).
\end{equation*}
Similarly there exists a unique solution $u$ to \eqref{PE} such that
\begin{equation*}
 u \in L^\infty([0,T],\mathcal{G}^s_{r,\tau(\,\cdot\,)}).
\end{equation*}
\end{theorem}

\begin{proof}
The method of proof of existence is based on Galerkin approximations and the priori estimate in Proposition \ref{proposition1}. For a fixed positive integer $n$, we will look for a sequence of functions $u^\nu_n(t,\cdot)$ with $n\in \mathbb{N}$ in the form
$$
 u^\nu_n(t,x)=\sum_{j=1}^n \alpha_{j,n}^\nu(t) \omega_j,
$$
where $\{\omega_j\}_{j=1}^\infty$ are the orthonormal basis in Proposition \ref{prop4}.
Let $W_n$ be the space spanned by $\{w_1,w_2,\cdots,w_n\}$,
and $\mathbb{\chi}_n$ is the orthogonal projector in $\mathcal{L}^2$ into $W_n$.
The approximating equation is as follows,
\begin{equation}\label{3.18}
\left\{
\begin{aligned}
\frac{d}{dt}u^\nu_n+\nu A u^\nu_n+\mathbb{\chi}_n\mathbb{P}(u^\nu_n\cdot\nabla u^\nu_n) &=0,\ \\
u^\nu_n|_{t=0} &=\mathbb{\chi}_n a
\end{aligned}
\right.
\end{equation}
Taking the $L^2$ inner product with $w_j, j=1,2,\ldots, n$, then the equation system \eqref{3.18} is equivalent with the following ordinary differential equation system,
\begin{equation}\label{3.19}
\left\{
\begin{aligned}
\frac{d}{dt} \alpha_{j,n}^\nu(t) +\nu \lambda_j \alpha_{j,n}^\nu +\sum_{k,\ell}b(\omega_k,\omega_\ell,\omega_j) \alpha_{k,n}^\nu \alpha_{\ell, n}^\nu &=0,\quad\,j=1,2,\ldots,n,\\
 \alpha_{j,n}^\nu(0) &=\inner{a,\omega_j},
\end{aligned}
\right.
\end{equation}
where $b(\omega_k,\omega_\ell,\omega_j)=\inner{\omega_k\cdot\nabla\omega_\ell,\omega_j}$ satisfying $b(\omega_k,\omega_\ell,\omega_j)=-b(\omega_k,\omega_j,\omega_\ell)$. By the standard ordinary differential equation theory, there exists a solution to \eqref{3.19} local in time interval $[0,T_n)$. In order to show that $T_n$ can be extended to $T$, we multiply with $\alpha_{j,n}^\nu(t)$ on both sides of \eqref{3.19} and take sum over $1\leq j\leq n$. We have
\begin{equation}\label{3.20}
\frac{1}{2}\frac{d}{dt} \left(\sum_{j=1}^n \alpha_{j,n}^\nu(t)^2\right)+\nu\lambda_j \sum_{j}\alpha_{j,n}^\nu(t)^2 = 0,
\end{equation}
because
$$
\sum_{j,k,\ell}b(\omega_k,\omega_\ell,\omega_j)\alpha^\nu_{k,n}\alpha^\nu_{\ell,n}\alpha^\nu_{j,n}=0.
$$
Moreover, from \eqref{3.20}, we have  
$$
 \norm{u^\nu_n(t,\cdot)}_{L^2}=\sum_{j=1}^n \alpha_{j,n}^\nu(t)^2\leq \sum_{j=1}^n \alpha_{j,n}^\nu(0)^2 \leq \norm{a}_{L^2}^2,\quad \forall t>0
$$
Then we have for every $T_n$, it can be extended to arbitrary large, so it can be extended to $T$. And we also obtain
\begin{equation*}
 u^\nu_n \ \text{remains bounded in}\ L^\infty\left([0,T];L^2\right),\quad \forall n.\
\end{equation*}
Moreover, we obtain a solution $u^\nu_n(t,x)=\sum_{j=1}^n \alpha_{j,n}^\nu(t)w_j(x)$ to \eqref{3.18} and we know that $u_n^\nu(t,x)\in \mathcal{G}^{s}_{r+1,\tau(t)}$ for $0<t<T$ because it is only finite sum of $w_j$ for fixed $n$. We then want to obtain the uniform Gevrey class norm bound for $u^\nu_n$. To this end, we first apply $\Lambda^r e^{\tau \Lambda^{1/s}}$ on both sides of \eqref{3.18} and then take the $L^2$ inner product with $\Lambda^r e^{\tau\Lambda^{1/s}}$ to obtain
\begin{equation*}
\begin{split}
\frac{1}{2}\frac{d}{dt} &\norm{u^\nu_n(t,\cdot)}_{\mathcal{G}^{s}_{r,\tau(t)}}^2+
\nu\norm{u^\nu_n(t,\cdot)}_{\mathcal{G}^{s}_{r+1,\tau}}^2\\
&=\tau^\prime(t) \norm{u^\nu_n(t,\cdot)}_{\mathcal{G}^s_{r+\frac{1}{2s},\tau}}^2-\inner{\Lambda^r e^{\tau\Lambda^{1/s}}\mathbb{\chi}_n\mathbb{P}(u^\nu_n\cdot\nabla u^\nu_n),\Lambda^{r} e^{\tau(t) \Lambda^{1/s}}u^\nu_n}.\\
\end{split}
\end{equation*}
We note that the operator $\mathbb{\chi}_n$ and $\mathbb{P}$ commute with $\Lambda^r e^{\tau\Lambda^{1/s}}$, and they are symmetric, then
$$
 \inner{\Lambda^r e^{\tau\Lambda^{1/s}}\mathbb{\chi}_n\mathbb{P}(u^\nu_n\cdot\nabla u^\nu_n),\Lambda^{r} e^{\tau(t) \Lambda^{1/s}}u^\nu_n}=\inner{\Lambda^r e^{\tau\Lambda^{1/s}}(u^\nu_n\cdot\nabla u^\nu_n),\Lambda^{r} e^{\tau(t) \Lambda^{1/s}}u^\nu_n}.
$$
With the arguments in Proposition \ref{proposition1}, we have,  
\begin{equation*}
\norm{u_\nu^n(t,\cdot)}_{\mathcal{G}^s_{r,\tau(t)}}\leq G_T,\quad 0<t\leq T<T^\ast,\quad \forall n.
\end{equation*}
Thus
\begin{equation}\label{3.23}
 u^\nu_n\,\,\text{remains bounded in}\,L^\infty([0,T];\mathcal{G}^s_{r,\tau(\,\cdot\,)}).
\end{equation}
In order to pass to the limit in the nonlinear term using a compactness theorem, we need to estimate on $\frac{d u_n^\nu}{dt}$. From \eqref{3.18}, we have
\begin{equation*}
 \frac{d u^\nu_n}{dt} =-\nu A u^\nu_n- \mathbb{\chi}_n\mathbb{P} (u^\nu_n\cdot\nabla u^\nu_n).
\end{equation*}
Then we obtain
\begin{equation*}
\begin{aligned}
 \norm{\frac{d u^\nu_n}{dt}}_{L^2} &\leq \norm{\mathbb{\chi}_n\mathbb{P} (u^\nu_n\cdot\nabla u^\nu_n)}_{L^2}+\norm{\nu A u^\nu_n}_{L^2}\\
                                   &\leq \norm{u^\nu_n\cdot\nabla u^\nu_n}_{L^2}+\nu_0\norm{u^\nu_n}_{H^r}\\
                                   &\leq C\norm{u^\nu_n}_{H^r}^2+\nu_0\norm{u^\nu_n}_{H^r}.
\end{aligned}
\end{equation*}
We recall that
$$
 \norm{u^\nu_n(t,\cdot)}_{H^r}\leq C_T,\quad\,0<t\leq T,\quad\,\forall n.
$$
So we obtain
\begin{equation}\label{3.25}
 \frac{d}{dt}u^\nu_m\, \text{remains bounded in}\,L^\infty([0,T];L^2).
\end{equation}
By \eqref{3.23} and \eqref{3.25}, noting that $H^r(\mathbb{T}^3)$ is compactly embedded in $L^2(\mathbb{T}^3)$ from Rellich-Kondrachov Compactness Theorem in \cite{EV}, a compactness theorem in \cite{JL,RT} indicates the existence of the limit $u^\nu \in L^\infty([0,T];\mathcal{G}^s_{r,\tau(\,\cdot\,)})$ of a subsequence of $u^\nu_m$ such that
\begin{equation}\label{3.26}
\left\{
\begin{aligned}
 \frac{d}{dt}\inner{u^\nu(t),v}-\nu\inner{\Delta u^\nu(t),v}+\inner{u^\nu(t)\cdot\nabla u^\nu(t),v} &=0,\ \ \forall v\in \mathcal{L}^2\\
  u^\nu(0) &=a
\end{aligned}
\right.
\end{equation}
For Euler equations, one can take very similar approach to obtain the existence of solution in Gevrey class space and we omit the details here. Thus we prove Proposition \ref{existence}.
\end{proof}
It remains to show that $u^\nu$ is the solution of \eqref{1.1}. In fact it can be obtained from \eqref{3.26} that
\begin{equation*}
 \mathbb{P}\left\{\frac{d}{dt} u^\nu(t) -\nu \Delta u^\nu(t)+ u^\nu(t)\cdot\nabla u^\nu(t)\right\}=0.
\end{equation*}
So there exists a scalar function $p^\nu$ such that
\begin{equation*}
 \frac{d}{dt} u^\nu(t) -\nu \Delta u^\nu(t)+ u^\nu(t)\cdot\nabla u^\nu(t)+\nabla p^\nu=0,
\end{equation*}
where $p^\nu$ is unique up to a constant, and $p^\nu$ satisfies
\begin{equation}\label{new2}
\begin{aligned}
 -\Delta p^\nu =\nabla\cdot(u^\nu\cdot\nabla u^\nu),
\end{aligned}
\end{equation}
with periodic boundary condition. For the regularity of the pressure $p^\nu(t,x)$, we have the following Proposition.

\begin{proposition}\label{proposition}
Let $p^\nu$ satisfies \eqref{new2} , then the following estimate holds,
\begin{equation*}
\begin{aligned}
 \norm{p^\nu(t,\cdot)}_{\mathcal{G}^s_{r+1,\tau}}\leq CG_T^2, \quad\,0<t\leq T.
\end{aligned}
\end{equation*}
And for the pressure $p(t,x)$ in \eqref{1.2}, we also have
\begin{equation*}
  \norm{p(t,\cdot)}_{\mathcal{G}^s_{r+1,\tau}}\leq CG_T^2, \quad\,0<t\leq T\, ,
\end{equation*}
where $C,T,G_T$ are defined in Proposition \ref{proposition1}.
\end{proposition}
\begin{proof}
To study the pressure $p^\nu$, the existence is obvious results from standard elliptic equation theory. We consider the regularity of $p^\nu(t,x)$. First we apply the operator $\Lambda^r e^{\tau\Lambda^{1/s}}$ on both sides of \eqref{new2} and then take $L^2$ inner product with $\Lambda^{r}e^{\tau \Lambda^{1/s}}p^\nu$ to obtain
\begin{equation}\label{3.15}
\begin{split}
-\inner{\Delta p^\nu,\Lambda^{2r}e^{2\tau \Lambda^{1/s}}p^\nu}=\inner{\nabla\cdot(u^\nu\cdot\nabla u^\nu),\Lambda^{2r}e^{2\tau \Lambda^{1/s}}p^\nu}.
\end{split}
\end{equation}
Here if we can write $p^\nu(t,x)=\sum_{j\in\mathbb{Z}^3}\hat p^{\nu}_{j}e^{i j\cdot x}$, then the left side of \eqref{3.15} is
$$
-\inner{\Delta p^\nu,\Lambda^{2r}e^{2\tau \Lambda^{1/s}}p^\nu}= \sum_{j\in\mathbb{Z}^3} \abs{j}^{2r+2} e^{2\tau \abs{j}^{1/s}} \abs{\hat p^\nu_j}^2=\norm{p^\nu}_{\mathcal{G}^s_{r+1,\tau}}^2.
$$
The right hand side of \eqref{3.15} can be bounded by
\begin{equation*}
\begin{split}
   &\quad \abs{\inner{\nabla\cdot(u^\nu\cdot\nabla u^\nu),\Lambda^{2r}e^{2\tau \Lambda^{1/s}}p^\nu}}\\
   &=\abs{(2\pi)^3\sum_{k\in\mathbb{Z}^3}\sum_{j\in\mathbb{Z}^3}\abs{k}^{r-1}\left[\hat u^{\nu}_{j} \cdot (k-j)\right](k \cdot \hat u^\nu_{k-j})\hat p^\nu_{-k} \abs{k}^{r+1}e^{2\tau \abs{k}^{1/s}}}\\
   &=\abs{(2\pi)^3\sum_{k\in\mathbb{Z}^3}\sum_{j\in\mathbb{Z}^3}\abs{k}^{r-1}\left[\hat u^{\nu}_{j} \cdot (k-j)\right](j \cdot \hat u^\nu_{k-j})\hat p^\nu_{-k} \abs{k}^{r+1}e^{2\tau \abs{k}^{1/s}}}\\
   &\leq C\norm{p^\nu}_{\mathcal{G}^s_{r+1,\tau}}\\
   &\times\left\{\sum_{k\in\mathbb{Z}^3}\left[\sum_{j\in\mathbb{Z}^3}(\abs{j}^{r-1}+\abs{k-j}^{r-1})\abs{j}\abs{k-j}\abs{\hat u^\nu_{j}}e^{\tau \abs{j}^{1/s}}e^{\tau \abs{k}^{1/s}}\abs{\hat u^\nu_{k-j}}\right]^2\right\}^{1/2}\\
   &\leq C \norm{u^\nu}_{\mathcal{G}^s_{r,\tau}}^2 \norm{p^\nu}_{\mathcal{G}^s_{r+1,\tau}},
\end{split}
\end{equation*}
where $C$ is a constant depending on $r$. Then from above estimate, we obtain
\begin{equation*}
\norm{p^\nu}_{\mathcal{G}^s_{r+1,\tau}} \leq C\norm{u^\nu}_{\mathcal{G}^s_{r,\tau}}^2.
\end{equation*}
From \eqref{3.6}, we obtain
\begin{equation*}
\norm{p^\nu(t)}_{\mathcal{G}^s_{r+1,\tau}} \leq C A(t)^2\leq CG_T^2,\ t\in[0,T].
\end{equation*}
For the pressure $p(t,x)$ of Euler equation \eqref{1.2}, one can first obtain the following elliptic equation,
\begin{equation*}
\begin{aligned}
 -\Delta p =\nabla\cdot(u\cdot\nabla u).\\
\end{aligned}
\end{equation*}
Then using the same arguments as above, one can obtain the same results for $p$.
\end{proof}

\section{Convergence of solutions in Gevrey space}\label{section4}

In the previous Section we have proved the existence of solutions to the Navier-Stokes equation and Euler equation in Gevrey class space. In this Section we will show the vanishing viscosity limit of Navier-Stokes equation in Gevrey class space.  Moreover, we can obtain the converging rate with respect to $\nu$.

\begin{theorem}
Let $u^\nu, p^\nu$ and $u, p$ are the solutions we have obtained in the previous Section, where
 $$
  u^\nu,u\in L^\infty([0,T],\mathcal{G}^s_{r,\tau(\,\cdot\,)}),\quad p^\nu,p\in L^\infty([0,T],\mathcal{G}^s_{r+1,\tau(\,\cdot\,)}).
 $$
Then the following estimates hold,
\begin{equation}\label{new4}
 \norm{u^\nu(t,\cdot)-u(t,\cdot)}_{\mathcal{G}^s_{r-1,\tau(t)}}\leq C\sqrt{\nu},\quad \norm{p^\nu(t,\cdot)-p(t,\cdot)}_{\mathcal{G}^s_{r,\tau(t)}}\leq C\sqrt{\nu},
\end{equation}
for any $0<t\leq T$, where $C$ is a constant depending on $r,s,a,T$.
\end{theorem}
\begin{proof}
Let us first consider the new equation for $w=(u^\nu-u)$ and $\tilde{p}=p^\nu-p$,
\begin{equation}\label{3.28}
\left\{
\begin{split}
\frac{\partial w}{\partial t}-\nu \Delta u^\nu+ w\cdot\nabla u^\nu+u\cdot\nabla w+\nabla\tilde{p} &=0,\\
\nabla\cdot w &=0,\\
w|_{t=0}      &=0.
\end{split}
\right.
\end{equation}
Then we apply the operator $\Lambda^{r-1}e^{\tau\Lambda^{1/s}}$ on both sides of \eqref{3.28} and take the $L^2$ inner product with $\Lambda^{(r-1)}e^{\tau \Lambda^{1/s}} w$ on both sides to obtain,
\begin{equation}\label{3.29}
\begin{split}
\frac{1}{2}\frac{d}{dt}\norm{w(t)}_{\mathcal{G}^{s}_{r-1,\tau}}^2 &=\nu\inner{\Lambda^{(r-1)}e^{\tau \Lambda^{1/s}}\Delta u^\nu,\Lambda^{(r-1)}e^{\tau \Lambda^{1/s}}w}+{\tau^\prime}\norm{w}_{\mathcal{G}^s_{r-1+\frac{1}{2s},\tau}}^2\\
&\quad
-\inner{\Lambda^{(r-1)}e^{\tau \Lambda^{1/s}}(w\cdot\nabla u^\nu),\Lambda^{(r-1)}e^{\tau \Lambda^{1/s}} w}\\
&\quad-\inner{\Lambda^{(r-1)}e^{\tau \Lambda^{1/s}}(u\cdot\nabla w),\Lambda^{(r-1)}e^{\tau \Lambda^{1/s}} w},
\end{split}
\end{equation}
where the term $\inner{\Lambda^{(r-1)}e^{\tau \Lambda^{1/s}}\nabla\tilde{p},\, \Lambda^{(r-1)}e^{\tau \Lambda^{1/s}}w}$ disappeares since $w=u^\nu-u$ is divergence free. It remains to estimate the right hand side of \eqref{3.29}, for convenience, we denote
\begin{equation*}
\begin{aligned}
 \mathcal{I}_1 &=\nu\inner{\Lambda^{(r-1)}e^{\tau \Lambda^{1/s}}\Delta u^\nu,\Lambda^{(r-1)}e^{\tau\Lambda^{1/s}}w},\\
 \mathcal{I}_2 &=\inner{\Lambda^{(r-1)}e^{\tau \Lambda^{1/s}}(w\cdot\nabla u^\nu),\Lambda^{(r-1)}e^{\tau\Lambda^{1/s}}w},\\
 \mathcal{I}_3 &=\inner{\Lambda^{(r-1)}e^{\tau \Lambda^{1/s}}(u\cdot\nabla w),\Lambda^{(r-1)}e^{\tau\Lambda^{1/s}}w}.
\end{aligned}
\end{equation*}
Using the discrete H\"older inequality, one can obtain
\begin{equation}\label{3.30}
\begin{aligned}
 \abs{\mathcal{I}_1} &=\nu\abs{\inner{\Lambda^{(r-1)}e^{\tau \Lambda^{1/s}}\Delta u^\nu,\Lambda^{(r-1)}e^{\tau \Lambda^{1/s}}w}} \\
  &=\nu\abs{-(2\pi)^3\sum_{k\in\mathbb{Z}^3} \abs{k}^{2r}e^{2\tau \abs{k}^{1/s}}(\hat u^\nu_k\cdot \hat w_{-k})}\\
  &\leq \nu(2\pi)^3\sum_{k\in\mathbb{Z}^3} \abs{k}^{r+1}e^{\tau \abs{k}^{1/s}}\abs{\hat u^\nu_k}\abs{k}^{r-1}e^{\tau \abs{k}^{1/s}}\abs{\hat w_{-k}}\\
  &\leq C\nu\norm{u^\nu}_{\mathcal{G}^s_{r+1,\tau}}\norm{w}_{\mathcal{G}^s_{r-1,\tau}}.
\end{aligned}
\end{equation}
Then we have
\begin{equation*}
 \abs{\mathcal{I}_1} \leq C\nu\norm{u^\nu}_{\mathcal{G}^s_{r+1,\tau}}\norm{w}_{\mathcal{G}^s_{r-1,\tau}}.
\end{equation*}
As for $\mathcal{I}_2$, we first write it into the sum of their Fourier coefficients,
\begin{equation*}
\begin{split}
  \mathcal{I}_2=i(2\pi)^3\sum_{k\in\mathbb{Z}^3}\sum_{j\in\mathbb{Z}^3}\left[\hat w_j\cdot(k-j)\right](\hat u^\nu_{k-j}\cdot\hat w_{-k})\abs{k}^{2(r-1)}e^{2\tau\abs{k}^{1/s}}.\\
\end{split}
\end{equation*}
Since $r>9/2$, there exists a constant C such that
$$
\abs{k}^{r-1}\leq C\left(\abs{j}^{r-1}+\abs{k-j}^{r-1}\right),
$$
and $s\geq1$ implies
$$
e^{\tau\abs{k}^{1/s}}\leq e^{\tau\abs{j}^{1/s}} e^{\tau\abs{k-j}^{1/s}}.
$$
Thus $\mathcal{I}_2$ can be bounded by
\begin{equation*}
\begin{aligned}
 \abs{\mathcal{I}_2} &\leq C\sum_{k\in\mathbb{Z}^3}\sum_{j\in\mathbb{Z}^3}\bigg[(\abs{j}^{r-1}+\abs{k-j}^{r-1})\abs{\hat w_j}\abs{k-j}\abs{\hat u^\nu_{k-j}}e^{\tau\abs{j}^{1/s}}e^{\tau\abs{k-j}^{1/s}}\\
 &\quad\times \abs{\hat w_{-k}}\abs{k}^{r-1}e^{\tau\abs{k}^{1/s}} \bigg].
\end{aligned}
\end{equation*}
Then by discrete H\"older inequality and Minkowski inequality, we obtain
\begin{equation}\label{3.31}
 \abs{\mathcal{I}_2}\leq C\norm{u^\nu}_{\mathcal{G}^s_{r,\tau}}\norm{w}_{\mathcal{G}^s_{r-1,\tau}}^2.
\end{equation}
As for $\mathcal{I}_3$, where
\begin{equation*}
\begin{aligned}
 \mathcal{I}_3=i(2\pi)^3\sum_{k\in\mathbb{Z}^3}\sum_{j\in\mathbb{Z}^3}[\hat u_j\cdot(k-j)](\hat w_{k-j}\cdot\hat w_{-k})\abs{k}^{2(r-1)}e^{2\tau\abs{k}^{1/s}}.
\end{aligned}
\end{equation*}
Here again the cancellation property implies
\begin{equation*}
\begin{aligned}
0 &=\inner{u\cdot\nabla \Lambda^{r-1}e^{\tau\Lambda^{1/s}}w,\Lambda^{r-1}e^{\tau\Lambda^{1/s}}w}\\
  &=i(2\pi)^3\sum_{k\in\mathbb{Z}^3}\sum_{j\in\mathbb{Z}^3}[\hat u_j\cdot(k-j)](\hat w_{k-j}\cdot\hat w_{-k})\abs{k-j}^{r-1}e^{\tau\abs{k-j}^{1/s}}\abs{k}^{r-1}e^{\tau\abs{k}^{1/s}}.
\end{aligned}
\end{equation*}
Then we have
\begin{equation*}
\begin{split}
\mathcal{I}_3 &=i(2\pi)^3\sum_{k\in\mathbb{Z}^3}\sum_{j\in\mathbb{Z}^3}[\hat u_j\cdot(k-j)](\hat w_{k-j}\cdot\hat w_{-k})\abs{k}^{2(r-1)}e^{2\tau\abs{k}^{1/s}}\\
 &\quad -i(2\pi)^3\sum_{k\in\mathbb{Z}^3}\sum_{j\in\mathbb{Z}^3}[\hat u_j\cdot(k-j)](\hat w_{k-j}\cdot\hat w_{-k})\abs{k-j}^{r-1}e^{\tau\abs{k-j}^{1/s}}\abs{k}^{(r-1)}e^{\tau\abs{k}^{1/s}}\\
 &=\mathcal{R}_1+\mathcal{R}_2,
\end{split}
\end{equation*}
where we denote
\begin{equation*}
\begin{aligned}
\mathcal{R}_1 =i(2\pi)^3\sum_{k\in\mathbb{Z}^3}\sum_{j\in\mathbb{Z}^3} &\bigg\{[\hat u_j\cdot(k-j)](\hat w_{k-j}\cdot\hat w_{-k})\left(\abs{k}^{r-1}-\abs{k-j}^{r-1}\right)e^{\tau\abs{k}^{1/s}}\\
&\quad \times\abs{k}^{r-1}e^{\tau\abs{k}^{1/s}}\bigg\},
\end{aligned}
\end{equation*}
and
\begin{equation*}
\begin{aligned}
\mathcal{R}_2= i(2\pi)^3\sum_{k\in\mathbb{Z}^3}\sum_{j\in\mathbb{Z}^3}& \bigg\{[\hat u_j\cdot(k-j)](\hat w_{k-j}\cdot\hat w_{-k})\abs{k-j}^{r-1}\left(e^{\tau\abs{k}^{1/s}}-e^{\tau\abs{k-j}^{1/s}}\right)\\
 &\quad \times\abs{k}^{r-1}e^{\tau\abs{k}^{1/s}} \bigg\}.
\end{aligned}
\end{equation*}
Here we used a different strategy in the split of $\abs{k}^{r-1}e^{\tau\abs{k}^{1/s}}-\abs{k-j}^{r-1}e^{\tau\abs{k-j}^{1/s}}$ as in Lemma \ref{Lemma3} to estimate $\mathcal{R}_1$ and $\mathcal{R}_2$. With use of the following mean value theorem, there exists a constant $\theta\in(0,1)$ such that
\begin{equation*}
\begin{aligned}
 \abs{\abs{k}^{r-1}-\abs{k-j}^{r-1}} &=\abs{(r-1)(\abs{k}-\abs{k-j})\big[\theta\abs{k}+(1-\theta)\abs{k-j} \big]^{r-2}}\\
                               &\leq C\abs{j}(\abs{k}^{r-2}+\abs{k-j}^{r-2}).
\end{aligned}
\end{equation*}
Then by discrete H\"older inequality and Minkowski inequality, we have
\begin{equation*}
\begin{aligned}
\abs{\mathcal{R}_1} &\leq C\sum_{k\in\mathbb{Z}^3}\sum_{j\in\mathbb{Z}^3}\bigg[\abs{\hat u_j}\abs{k-j}\abs{\hat w_{k-j}}\abs{\hat w_{-k}}\abs{j}(\abs{j}^{r-2}+\abs{k-j}^{r-2})\\
 &\quad \times e^{\tau \abs{j}^{1/s}}e^{\tau\abs{k-j}^{1/s}}\abs{k}^{r-1}e^{\tau\abs{k}^{1/s}}\bigg]\\
 &\leq C\norm{u}_{\mathcal{G}^s_{r,\tau}}\norm{w}_{\mathcal{G}^s_{r-1,\tau}}^2.
\end{aligned}
\end{equation*}
As for $\mathcal{R}_2$, we use the inequality $\abs{e^\xi-1}\leq \abs{\xi}e^{\abs{\xi}}$ for $\forall \xi\in\mathbb{R}$ and Lemma \ref{Lemma2},
\begin{equation*}
\begin{aligned}
 &\abs{\mathcal{R}_2}\leq C\sum_{k\in\mathbb{Z}^3}\sum_{j\in\mathbb{Z}^3}\left[ \abs{\hat u_j} \abs{\hat w_{k-j}} \abs{\hat w_{-k}} \abs{k-j}^r e^{\tau\abs{k-j}^{1/s}}\abs{e^{\tau (\abs{k}^{1/s}-\abs{k-j}^{1/s})}-1}\right.\\
 &\quad\left.\times \abs{k}^{r-1}e^{\tau \abs{k}^{1/s}}\right]\\
 & \leq C\tau \sum_{k\in\mathbb{Z}^3}\sum_{j\in\mathbb{Z}^3}\bigg[ \abs{\hat u_j} \abs{\hat w_{k-j}} \abs{\hat w_{-k}} \abs{k-j}^{r-1} e^{\tau\abs{k-j}^{1/s}}e^{\tau \abs{j}^{1/s}} \frac{\abs{j}\abs{k-j}}{\abs{k}^{1-\frac{1}{s}}+\abs{k-j}^{1-\frac{1}{s}}}\\
 &\quad \times \abs{k}^{r-1}e^{\tau\abs{k}^{1/s}}\bigg].
\end{aligned}
\end{equation*}
For here we use the inequality $\abs{k-j}\leq 2\abs{k}\abs{j}$ for $k,j\neq 0$, then we have
\begin{equation*}
 \begin{aligned}
   \frac{\abs{k-j}}{\abs{k}^{1-\frac{1}{s}}+\abs{k-j}^{1-\frac{1}{s}}} &\leq \abs{k-j}^{1/s}\\
                                                                       &\leq C\abs{k-j}^{\frac{1}{2s}}\abs{k}^{\frac{1}{2s}}
                                                                        \abs{j}^{\frac{1}{2s}},
 \end{aligned}
\end{equation*}
where $C$ is a constant depending on $s$. Thus $\mathcal{R}_2$ can be bounded as follows,
\begin{equation*}
\begin{aligned}
 \abs{\mathcal{R}_2} &\leq C\tau \sum_{k\in\mathbb{Z}^3}\sum_{j\in\mathbb{Z}^3}\left[ \abs{\hat u_j} \abs{\hat w_{k-j}} \abs{\hat w_{-k}} \abs{k-j}^{r-1+\frac{1}{2s}} e^{\tau\abs{k-j}^{1/s}}e^{\tau \abs{j}^{1/s}} \abs{j}^{1+\frac{1}{2s}}\right.\\
 &\quad \left.\times \abs{k}^{r-1+\frac{1}{2s}}e^{\tau\abs{k}^{1/s}}\right]\\
 &\leq C\tau \sum_{k\in\mathbb{Z}^3}\sum_{j\in\mathbb{Z}^3}\left[ \abs{\hat u_j} \abs{\hat w_{k-j}} \abs{\hat w_{-k}} \abs{k-j}^{r-1+\frac{1}{2s}} e^{\tau\abs{k-j}^{1/s}}(1+\tau \abs{j}^{1/s}e^{\tau \abs{j}^{1/s}}) \right.\\
 &\quad \left.\times \abs{j}^{1+\frac{1}{2s}}\abs{k}^{r-1+\frac{1}{2s}}e^{\tau\abs{k}^{1/s}}\right]\\
 &\leq C\tau\norm{u}_{H^r}\norm{w}_{\mathcal{G}^s_{r-1+\frac{1}{2s},\tau}}^2+
  C\tau^2\norm{u}_{\mathcal{G}^s_{r,\tau}}\norm{w}_{\mathcal{G}^s_{r-1+\frac{1}{2s},\tau}}^2,
\end{aligned}
\end{equation*}
where we use the inequality $e^\xi\leq 1+\xi e^\xi$ for $\forall \xi\in\mathbb{R}^+$ with respect to $e^{\tau\abs{j}^{1/s}}$ and also the discrete H\"older inequality and Minkowski inequality in the above inequality.
Then we have
\begin{equation}\label{3.32}
\begin{split}
\abs{\mathcal{I}_3} \leq C\norm{u}_{\mathcal{G}^s_{r,\tau}}\norm{w}_{\mathcal{G}^s_{r-1,\tau}}^2+
C\tau\norm{u}_{H^r}\norm{w}_{\mathcal{G}^s_{r-1+\frac{1}{2s},\tau}}^2+C\tau^2\norm{u}_{\mathcal{G}^s_{r,\tau}}
\norm{w}_{\mathcal{G}^s_{r-1+\frac{1}{2s},\tau}}^2.
\end{split}
\end{equation}
Substituting \eqref{3.30}, \eqref{3.31} and \eqref{3.32} into \eqref{3.29}, we obtain
\begin{equation}\label{3.33}
\begin{split}
\frac{1}{2}\frac{d}{dt}\norm{w}_{\mathcal{G}^s_{r-1,\tau}}^2 &\leq \nu\norm{u^\nu}_{\mathcal{G}^s_{r+1,\tau}}\norm{w}_{\mathcal{G}^s_{r-1,\tau}}+
C\norm{u}_{\mathcal{G}^s_{r,\tau}}\norm{w}_{\mathcal{G}^s_{r-1,\tau}}^2\\
&\quad+\left(\tau^\prime+C\tau\norm{u}_{H^r}+C\tau^2\norm{u}_{\mathcal{G}^s_{r,\tau}}\right)
\norm{w}_{\mathcal{G}^s_{r-1+\frac{1}{2s},\tau}}^2.\\
\end{split}
\end{equation}
By the choice of $\tau$ in \eqref{3.10}, and noting that \eqref{3.6},\eqref{3.7},\eqref{3.8},\eqref{3.9} also hold for Euler equation \eqref{1.2}. Then choosing the appropriate constant $C$, one has $\tau^\prime+C\tau\norm{u}_{H^r}+C\tau^2\norm{u}_{H^r}+C\tau^2\norm{u}_{\mathcal{G}^s_{r,\tau}}\leq 0$, then we can obtain from \eqref{3.33} and \eqref{Abound},
\begin{equation}\label{3.34}
\frac{d}{dt}\norm{w(t,\cdot)}_{\mathcal{G}^{s}_{r-1,\tau(t)}}\leq \nu\norm{u^\nu(t,\cdot)}_{\mathcal{G}^s_{r+1,\tau(t)}}+CG_T\norm{w(t,\cdot)}_{\mathcal{G}^s_{r-1,\tau(t)}},\quad 0<t\leq T.
\end{equation}
Since $w(0)=0$ and Grownwall's inequality, \eqref{3.34} implies
\begin{equation*}
\begin{split}
\norm{w(t,\cdot)}_{\mathcal{G}^{s}_{r-1,\tau(t)}} \leq e^{CG_T t}\int_0^t \nu\norm{u^\nu(s,\cdot)}_{\mathcal{G}^{s}_{r+1,\tau(s)}}ds,\quad 0<t\leq T.
\end{split}
\end{equation*}
Recalling from \eqref{3.13} we have for $0<t\leq T$,
$$
\int_0^t \nu \norm{u^\nu(s,\cdot)}_{\mathcal{G}^s_{r+1,\tau(s)}}^2ds\leq M_T,\ \ t\in (0,T].
$$
With H\"older inequality, we have
\begin{equation*}
\begin{split}
\int_0^t \nu \norm{u^\nu(s,\cdot)}_{\mathcal{G}^s_{r+1,\tau(s)}}ds &\leq \int_0^t \sqrt{\nu} \sqrt{\nu}\norm{u^\nu(s,\cdot)}_{\mathcal{G}^s_{r+1,\tau(s)}}ds\\
 &\leq \sqrt{\nu}t^{1/2} M_T^{1/2}. \\
\end{split}
\end{equation*}
Then we have
\begin{equation}\label{3.36}
\norm{w(t,\cdot)}_{\mathcal{G}^s_{r-1,\tau(t)}} \leq C\sqrt{\nu}M_T^{1/2}t^{1/2}e^{CG_T t},\ \ 0<t\leq T.
\end{equation}
This proves the first estimate of \eqref{new4} by arranging the constant. Then we want to estimate $p^\nu(t)-p(t)$ in the norm of $G^s_{r,\tau}$. To do so, we first take the divergence of both sides of \eqref{3.28} to obtain the following elliptic equation,
\begin{equation}\label{3.37}
-\Delta \tilde{p} = \nabla \cdot (w\cdot\nabla u^\nu) +\nabla \cdot (u\cdot\nabla w).
\end{equation}
Then we first apply the operator $\Lambda^{r-1}e^{\tau\Lambda^{1/s}}$ on both sides of \eqref{3.37} and then take the $L^2$ inner product with $\Lambda^{(r-1)}e^{\tau \Lambda^{1/s}}\tilde{p}$ on both sides to obtain,
\begin{equation}\label{3.38}
\begin{split}
(2\pi)^3\norm{\tilde{p}}_{\mathcal{G}^s_{r,\tau}}^2 &=\inner{\Lambda^{r-1}e^{\tau\Lambda^{1/s}}\nabla\cdot(w\cdot\nabla u_\nu),\Lambda^{(r-1)}e^{\tau\Lambda^{1/s}}\tilde{p}}\\
 &\quad+\inner{\Lambda^{r-1}e^{\tau\Lambda^{1/s}}\nabla\cdot(u\cdot\nabla w),\Lambda^{(r-1)}e^{\tau\Lambda^{1/s}}\tilde{p}}\\
 &=i^2(2\pi)^3\sum_{k\in\mathbb{Z}^3}\sum_{j\in\mathbb{Z}^3}[\hat w_j\cdot (k-j)](k\cdot{\hat u^\nu}_{k-j})\abs{k}^{2(r-1)}e^{2\tau \abs{k}^{1/s}}\hat{\tilde{p}}_{-k}\\
 &\quad+i^2(2\pi)^3\sum_{k\in\mathbb{Z}^3}\sum_{j\in\mathbb{Z}^3}[\hat u_j\cdot (k-j)](k\cdot\hat{w}_{k-j})\abs{k}^{2(r-1)}e^{2\tau \abs{k}^{1/s}}\hat{\tilde{p}}_{-k}\\
 &=\mathcal{P}_1+\mathcal{P}_2,
\end{split}
\end{equation}
where we denote
\begin{align*}
\mathcal{P}_1 &=i^2(2\pi)^3\sum_{k\in\mathbb{Z}^3}\sum_{j\in\mathbb{Z}^3}[\hat w_j\cdot (k-j)](k\cdot\hat u^\nu_{k-j})\abs{k}^{2(r-1)}e^{2\tau \abs{k}^{1/s}}\hat{\tilde{p}}_{-k},\\
\mathcal{P}_2 &=i^2(2\pi)^3\sum_{k\in\mathbb{Z}^3}\sum_{j\in\mathbb{Z}^3}[\hat u_j\cdot (k-j)](k\cdot\hat w_{k-j})\abs{k}^{2(r-1)}e^{2\tau \abs{k}^{1/s}}\hat{\tilde{p}}_{-k}.
\end{align*}
It remains to estimate $\mathcal{P}_1$ and $\mathcal{P}_2$ in \eqref{3.38}. For simplicity we only compute $\mathcal{P}_1$, since $\mathcal{P}_2$ can be estimated in the same way.
\begin{equation}\label{3.39}
\begin{split}
 \abs{\mathcal{P}_1} &=\abs{(2\pi)^3 \sum_{k\in\mathbb{Z}^3}\sum_{j\in\mathbb{Z}^3}[\hat w_j\cdot (k-j)](j\cdot \hat u^\nu_{k-j})\abs{k}^{2(r-1)}e^{2\tau \abs{k}^{1/s}}\hat{\tilde{p}}_{-k}}\\
  &\leq (2\pi)^3\sum_{k\in\mathbb{Z}^3}\sum_{j\in\mathbb{Z}^3}\left[ C(\abs{j}^{r-2}+\abs{k-j}^{r-2})e^{\tau \abs{j}^{1/s}}e^{\tau\abs{k-j}^{1/s}} \abs{k-j}\right.\\
  &\quad\left.\times\abs{j}\abs{\hat w_j}\abs{\hat u^\nu_{k-j}} \abs{k}^r e^{\tau\abs{k}^{1/s}}\abs{\hat{\tilde{p}}_{-k}}\right]\\
  &\leq C\norm{w}_{G^s_{r-1,\tau}}\norm{u^\nu}_{G^s_{r,\tau}}\norm{\tilde{p}}_{G^s_{r,\tau}},
\end{split}
\end{equation}
where we use the fact $k\cdot \hat u^\nu_{k-j}=j\cdot \hat u^\nu_{k-j}$ from the divergence free condition. And, similarly, we can obtain
\begin{equation}\label{3.40}
\begin{split}
 \abs{\mathcal{P}_2} &\leq (2\pi)^3 \sum_{k\in\mathbb{Z}^3}\sum_{j\in\mathbb{Z}^3}\abs{[\hat {u}_j\cdot (k-j)](j\cdot\hat{w}_{k-j})\abs{k}^{2(r-1)}e^{2\tau \abs{k}^{1/s}}\hat{\tilde{p}}_{-k}}\\
  &\leq C\norm{w}_{\mathcal{G}^s_{r-1,\tau}}\norm{u}_{\mathcal{G}^s_{r,\tau}}\norm{\tilde{p}}_{\mathcal{G}^s_{r,\tau}}.
\end{split}
\end{equation}
Substituting \eqref{3.39} and \eqref{3.40} into \eqref{3.38}, we obtain
\begin{equation}\label{3.41}
 \norm{\tilde{p}(t,\cdot)}_{\mathcal{G}^s_{r,\tau(t)}} \leq CG_T\norm{w(t,\cdot)}_{\mathcal{G}^s_{r-1,\tau(t)}},\quad 0<t\leq T.
\end{equation}
Then by \eqref{3.36} and \eqref{3.41} we have
$$
\norm{\tilde{p}(t,\cdot)}_{\mathcal{G}^s_{r,\tau(t)}}\leq C\sqrt{\nu}M_T^{1/2}t^{1/2}G_T e^{CG_T t}, \quad 0\leq t\leq T.
$$
This proves \eqref{new4} by arranging the constants. Thus we have proven Theorem \ref{Theorem1}.
\end{proof}

\bigskip
\noindent{\bf Acknowledgements.} The research of the second author was supported by NSF of China(11422106) and Fok Ying Tung Education Foundation (151001), the research of the first author and the last author is supported partially by
``The Fundamental Research Funds for Central Universities of China".

\end{document}